\documentclass[reqno,12pt]{amsart}
\usepackage{amsmath,epsfig}
\usepackage{amssymb}
\usepackage[ngerman]{babel}
\usepackage[latin1]{inputenc}
\usepackage[T1]{fontenc}
\usepackage{lmodern}
\usepackage{paralist} 

\newtheoremstyle{rem}{1.3ex}{1.3ex}{\rmfamily}{}
{\itshape\rmfamily}{}{1.5ex}{}

\newtheorem{theorem}{Theorem}[section]
\newtheorem{lemma}[theorem]{Lemma}
\newtheorem{prop}[theorem] {Proposition}

\theoremstyle{definition}

\newtheorem{remark}[theorem] {Remark}

\renewcommand{\section}{\secdef\sct\sect}
\newcommand{\sct}[2][default]{\refstepcounter{section}
\setcounter{equation}{0}
\vspace{0.5cm}
\centerline{ \large
\scshape \arabic{section}.\ #1}
\vspace{0.3cm}}
\newcommand{\sect}[1]{
\vspace{0.5cm}
\centerline{\large\scshape #1}
\vspace{0.3cm}}
\DeclareMathOperator{\arctanh}{arctanh}

\renewcommand{\subsection}{\secdef \subsct\sbsect}
\newcommand{\subsct}[2][default]{\refstepcounter{subsection}
\nopagebreak
\vspace{0.5\baselineskip}
{\flushleft\bf \arabic{section}.\arabic{subsection}~\bf #1  }
\nopagebreak}
\newcommand{\sbsect}[1]{\vspace{0.1cm}\noindent
{\bf #1}\vspace{0.1cm}}

\newcommand{\Ll}{\mathcal{L}}

\newcommand{\F}{{\mathcal {F}}}

\newcommand{\Zz}{\mathbb{Z} }
\newcommand{\Oo}{{\mathcal {O}}}
\newcommand{\G}{{\mathcal {G}}}
\newcommand{\Ei}{\mathbf{1}}

\newcommand{\B}     {\mathcal{B}}
\newcommand{\R}     {\mathbb{R}}

\newcommand{\N}     {\mathbb{N}}
\newcommand{\Ws}   {\mathbb{P}}

\newcommand{\E}     {\mathbb{E}}

\newcommand{\V}     {\mathbb{V}}

\def\1{{\mathchoice {1\mskip-4mu\mathrm l}
                    {1\mskip-4mu\mathrm l}
                    {1\mskip-4.5mu\mathrm l} {1\mskip-5mu\mathrm l}}}

\DeclareMathOperator{\sgn}{sign}

\begin{document}
\renewcommand{\proofname}{Proof}
\title[On rates of convergence for the overlap in the Hopfield model]{\large
On rates of convergence\\\vspace{2mm} for the overlap\\\vspace{5mm} in the Hopfield model}

\author[Peter Eichelsbacher and Bastian Martschink]{}
\maketitle
\thispagestyle{empty}
\vspace{0.2cm}

\centerline{\sc Peter Eichelsbacher\footnote{Ruhr-Universit\"at Bochum, Fakult\"at f\"ur Mathematik,
NA 3/66, D-44780 Bochum, Germany, {\tt peter.eichelsbacher@rub.de}} and Bastian Martschink\footnote{Hochschule Bonn-Rhein Sieg, Fachbereich 03,
B 295, D-53757 Sankt Augustin, Germany, {\tt bastian.martschink@h-brs.de} \\The authors have been supported by Deutsche Forschungsgemeinschaft via SFB/TR 12.}
}

\vspace{2 cm}


\begin{quote}
{\small {\bf Abstract:} 
We consider the Hopfield model with $n$ neurons and an increasing number $p=p(n)$ of randomly chosen patterns and use Stein's method
to obtain rates of convergence for the central limit theorem of overlap parameters, which holds for every fixed choice of the overlap parameter
for almost all realisations of the random patterns.}
\end{quote}





\setcounter{section}{0}

\bigskip
\bigskip
\bigskip

\section{Introduction}

\subsection{The Hopfield model}
The so-called Hopfield model was introduced by Figotin and Pastur in \cite{Figotin/Pastur:1977} and \cite{Figotin/Pastur:1978} as a model for a spin glass. They studied a class of spin glass models which also included the one with the energy function known today as the Hopfield model, which was also introduced by Hopfield in \cite{Hopfield:1982} in the context of neural networks as a model for an associative memory with $n\in\N$ neurons. Thus Hopfield linked the study of neural networks to the one of spin models. The success of this model was mainly based on this reinterpretation of the model and therefore it may be right to call it the Hopfield model. Being a model for the associate (also termed content-addressable) memory it is not derived directly from a physical or biological system. Roughly speaking, the recognition and/or retrieval of one out of $p\in\N$ stored patterns constitutes the central problem of the model. This means that one wants to store a certain amount of information and perform the quite difficult task to recognize it on the basis of partial or corrupted data, which is not easy for a usual search algorithm.\\
We consider a system of $n\in\N$ neurons. Each neuron can be in one of two possible states, either $-1$ or $1$. We will denote by $\sigma_i\in\{-1,1\}$ the neural activity of the $i^{\text{th}}$ neuron, $i\in\{1,\ldots,n\}$ and thus, in the context of spin systems, $\sigma_i$ would be the spin variable at $i\in\{1,\ldots,n\}$. Thus a spin configuration $(\sigma_1,\ldots,\sigma_n)$ is taken from the set of spin configurations $\{-1,1\}^n$. In general the instantaneous configuration of all the spin variables at a given time describes the state of such a network. Furthermore let $(\Omega,\B, \Ws)$ be an abstract probability space. The model consists of $p\in\N$ stored patterns on this space which will be denoted by $\xi^{\mu}$, $\mu\in\{1,\ldots,p\}$. Thus $\xi^{\mu}=(\xi_1^{\mu},\ldots,\xi_n^{\mu})\in\{-1,1\}^n$ describes the codification of the $\mu^{\text{th}}$ stored pattern. $(\sigma_i)_{i\in\N}$ and $(\xi_i^{\mu})_{i\in\N}$ with $\mu\in\N$ are considered to be random variables and we will assume that the family of random variables $\{\sigma_i,\xi_j^{\mu}\mid i,j,\mu\in\N\}$ is independent. Additionally we assume that the random variables satisfy $\Ws(\sigma_i=\pm 1)=1/2$ and $\Ws(\xi_j^{\mu}=\pm 1)=1/2$. Thus we denote by $\Ws_{\xi}=(\frac{1}{2}\delta_{-1}+\frac{1}{2}\delta_1)^{\otimes \N^2}$ the marginal distribution of the patterns $\xi=(\xi_i^{\mu})_{i,\mu\in\N}$, and similarly, by $\Ws_{\sigma}=(\frac{1}{2}\delta_{-1}+\frac{1}{2}\delta_1)^{\otimes \N}$ the marginal distribution of the spin variables $\sigma=(\sigma_i)_{i\in\N}$. As $n\rightarrow\infty$ $p$ can either be fixed or increasing with $n$. Now let
\begin{align}\label{HHM}
H_n(\sigma,\xi)=-\frac{1}{2n}\sum\limits_{\mu=1}^p\sum\limits_{i,j=1}^n\xi_i^{\mu}\xi_j^{\mu}\sigma_i\sigma_j,\text{   } n\in\N,
\end{align}
denote the Hopfield Hamiltonian. At this point one might notice the spin-flip dynamic
$H_n(-\sigma,\xi)=H_n(\sigma,\xi)$,
showing that the Hopfield model cannot distinguish between a spin configuration and its negative. Governed by this Hamiltonian, \cite{Amit:1985} presented a generalized Glauber single-spin dynamics on the set of spin configurations at finite temperature $1/\beta\in (0,\infty)$, which describes a reversible and irreducible Markov process. The equilibrium distribution of this process is the finite-volume Gibbs measure
\begin{eqnarray}
\text{d}P_{n,\beta,\xi}( \sigma)=\frac{1}{Z_{n,\beta,\xi}}\exp\left(-\beta H_n(\sigma,\xi)\right)\text{d}\Ws_{\sigma}\label{PNBXI},
\end{eqnarray}
where the partition function $Z_{n,\beta,\xi}$ is the appropriate normalization.\\
In the sequel the focus of attention will be on the investigation of the behavior of the so-called overlap under the equilibrium distribution $P_{n,\beta,\xi}$ as $n\rightarrow\infty$. Let \begin{align}\label{xidef}
\xi_i=(\xi_i^{\mu})_{\mu\in\{1,\ldots,p\}}\text{, }i\in\{1,\ldots,n\},
\end{align}
 be the vector consisting of the $i^{\text{th}}$ components of the first $p$ patterns. If $p$ is not constant and grows with $n$, $\xi_i\in\R^p$ still depends on $n$ via the dimension. We define the {\it overlap} by
\begin{align}\label{Hoverlap}
\frac{1}{n}S_n(\sigma,\xi)=\frac{1}{n}\sum\limits_{i=1}^n\xi_i\sigma_i \in\R^p,
\end{align}
with $\xi_i\sigma_i=(\xi_i^1\sigma_i,\ldots,\xi_i^p\sigma_i)^t$. 
With the overlap we obtain a comparison between the spin configuration $\sigma$ and the stored patterns $\xi^{\mu}$, $\mu\in\{1,\ldots,p\}$, meaning that the $\mu^{\text{th}}$ overlap parameter - the $\mu^{\text{th}}$ component of \eqref{Hoverlap} - equals one if and only if $\sigma_i=\xi_i^{\mu}$ for all $i\in\{1,\ldots,n\}$. Definition \eqref{Hoverlap} provides the opportunity to express the Hamiltonian \eqref{HHM} in a more convenient way. It can be rewritten as the quadratic function of the overlap
\begin{align*}
H_n(\sigma,\xi)=-\frac{n}{2}\bigl\Vert \frac{1}{n}S_n(\sigma,\xi)\bigr\Vert^2,
\end{align*}
where $\Vert\cdot\Vert$ denotes the Euclidean norm in $\R^p$. 
If there is no opportunity for confusion we will drop the explicit dependence on $\sigma$ and $\xi$ and write $S_n$ and $H_n$ instead of $S_n(\sigma,\xi)$ and
$H_n(\sigma,\xi)$, respectively. 

In the case $p=1$ the Hopfield model and the Curie-Weiss model are the same apart from a change of variable. 
The Curie-Weiss model is a well-known approximation of the Ising-model. The classical theory of magnetism occupies a central place in the physical literature. It allows the study of the behavior of thermodynamic quantities such as the specific heat, isothermal susceptibility, and magnetization in the neighborhood of the critical point. Because of its relative simplicity and the qualitative correctness of at least some of its predictions, it has been historically important. For our investigation of the Hopfield model we focus on the so called Curie-Weiss equation given by 
\begin{align}\label{CWE}
\beta x=\arctanh(x).
\end{align}
 This equation is also called {\it mean field} or {\it fixed point} equation. Its derivation can for example be found in \cite{Ellis:LargeDeviations}. Of course this equation may have many solutions. Let $x^{\pm}(\beta)$ denote for $\beta>0$ the largest (respectively smallest) solution $x\in (-1,1)$ of \eqref{CWE}. It was shown that $x^{+}(\beta)=-x^{-}(\beta)\neq 0$ for $\beta>\beta_c$, where $\beta_c=1$ is the critical inverse temperature. For $\beta\leq\beta_c$ we have $x^{\pm}(\beta)=0$. This definition of the Curie-Weiss equation can be extended to the case of the external magnetic field with strength $h\neq 0$ yielding
\begin{align}\label{CWEh}
\beta x+h=\arctanh(x).
\end{align}
Here let $x(\beta,h)$ denote the solution of \eqref{CWEh} which satisfies $\sgn(x)=\sgn(h)$.
As we will see these solutions of the Curie-Weiss equation discussed above play an important role when discussing the Hopfield model.
Abbreviate 
\begin{align*}
x^*:=\begin{cases}x^+(\beta),&\text{ if h=0},\\x(\beta,h),&\text{ otherwise}.\end{cases}
\end{align*}
For investigating the behaviour of the overlap, we also extend the notion of the Gibbs measure $P_{n,\beta,\xi}$ given in \eqref{PNBXI} to the case of an external magnetic field $he_l$ with strength $h\neq 0$ in the direction of the $l^{\text{th}}$ unit vector $e_l\in\R^p$. Thus, let
\begin{eqnarray}
\text{d}P_{n,\beta,he_l,\xi}( \sigma)=\frac{1}{Z_{n,\beta,he_l,\xi}}\exp\left(-\beta H_n+\langle S_n,he_l\rangle\right)\text{d}\Ws_{\sigma}\label{PNBHXI},
\end{eqnarray}
where $Z_{n,\beta,he_l,\xi}$ denotes the appropriate normalization.\\
For $\beta >0$ and $h\neq 0$ having the direction of the $l^{\text{th}}$ unit vector $e_l$ it was shown in \cite{Bovier/Gayrard/Picco:1994}  that for $\Ws_{\xi}$-almost all realizations of the patterns $\xi$ and if $p/n\rightarrow 0$ the overlap $\frac{S_n}{n}$ satisfies the {\it law of large numbers} 
\begin{align*}
P_{n,\beta,he_l,\xi} \biggl( \frac{S_n}{n} \in d\nu \biggr) 
\Rightarrow \delta_{\pm x(\beta,h)e_l} (d \nu)\text{ as }n \to \infty.
\end{align*}
The authors in \cite{Bovier/Gayrard/Picco:1994} stated that the condition on $p$ is the weakest possible under which the law of large numbers is satisfied. Note that for $\beta\leq\beta_c=1$ we have $x(\beta,h)=0$ and thus $\delta_0$ is the unique limiting measure in the high-temperature region. For $\beta>1$ it was mentioned that the measures of the law of large numbers are all distinct and they were referred to as so-called {\it extremal measures}. 

The corresponding {\it large deviation principle} (LDP for short) was established in \cite{Bovier/Gayrard:1996}. 
Under the assumption $p(n)/n \to 0$ for almost all $\xi$ the sequence
$(\frac{S_n}{n})_n$ under the Gibbs measure $P_{n, \beta, \xi}$ obeys a LDP with speed $n$ and {\it deterministic} rate function $I$. If the inverse temperature $\beta$ is different from the critical inverse temperature $\beta_c=1$ and $p(n)/n \to \infty$, the overlap parameter multiplied by $n^{\gamma}$ with $1/2<\gamma <1$ obeys a LDP with speed $n^{1 - \gamma}$ and a quadratic rate function, see \cite{Eichelsbacher/Loewe:2004}. The latter result is known as a moderate deviations principle (MDP for short).

On the scale of fluctuations, when analysing the distribution of $\sqrt{n} (S_n/n - x^* e_l)$, the disorder becomes visible. Indeed, for $p(n)/n \to 0$ and $(\beta,h) \not= (1,0)$ the overlap under $P_{n, \beta, \xi}$ satisfies $P_{\xi}$-almost surely a central limit theorem with a covariance matrix which could be expected from the analogy with the Curie-Weiss model and a centering which differs in the case $\beta >0$ or $h \not= 0$ from the naively expected one by a $\xi$-dependent
adjustment, see \cite{Gentz:1996b} and \cite{Bovier/Gayrard:1997}. In this paper we are aiming to give an alternative proof of these central limit theorems
for the overlap parameter under $P_{n, \beta, \xi}$. We will apply Stein's method. This method has emerged as a powerful tool for assessing the quality of
distributional approximations and it is notable for avoiding the use of transforms, and for supplying bounds, such as those of Berry-Esseen quality, on approximation error in the presence of dependence. We will be able to present rates of convergence for central limit theorems for the overlap parameter, which
are optimal for the Hopfield model with a finite number of randomly chosen patterns.
As in the Curie-Weiss model at the critical temperature $(\beta,h)=(1,0)$ the fluctuations are non Gaussian and the limiting distribution has a random component,
see \cite{Gentz/Loewe:1999} and \cite{Talagrand:2001}. Interesting enough the random term occurring in the central limit theorem is no longer present on a moderate deviations scale, where the overlap parameter has to be multiplied be $n^{\gamma}$ with $1/4 < \gamma < 1$: here for certain choices of $p(n)$
the rescaled overlap parameter obeys a MDP with speed $n^{1 - 4 \gamma}$ and a rate function that is basically a fourth power, see \cite{Eichelsbacher/Loewe:2004}. Anyhow, in this paper we do not consider the case $(\beta,h)=(1,0)$.

\subsection{Statement of the main results}
 
{\it General assumption.} From now on we make the assumption that $p=p(n)$, $p\leq n$ is a nondecreasing function of $n$ for all $n\in\N$. 

As in \cite{Gentz:1996} we choose a preferred pattern in two different ways. We consider the unbiased Hamiltonian \eqref{HHM} and investigate the fluctuations under the condition that the overlap is already in a neighbourhood of $x^* e_l$. Alternatively, the preferred pattern can be chosen by introducing the 
magnetic field as in \eqref{PNBHXI}. In the case of \eqref{HHM} with $\beta < \beta_c$ the central limit theorem holds with center zero. Otherwise the limit theorem requires a $\xi$-dependent adjustment of a deterministic centering. Therefore one has to control the influence of the random patterns.
For fixed $\epsilon>0$ we define 
\begin{align}
\alpha&:=\frac{1}{n}\max\biggl\{p,\left(\frac{3\log n}{\log(1+\epsilon)}\right)^4\biggr\},\nonumber\\ \label{epsilonn}
\epsilon_n&:=\sqrt{\alpha}(2+\sqrt{\alpha})(1+\epsilon).
\end{align}
By \cite[Proposition 2.1]{Gentz:1996} we see that the operator norm of
$
\Sigma^n(\xi)=\frac{1}{n}\sum\limits_{i=1}^n\xi_i\xi_i^t-\text{Id}_{\R^p}
$ 
converges to zero for $P_{\xi}$-almost all $\xi$: for $P_{\xi}$-almost all $\xi$, there exists an $n_0(\xi) \in \N$ such that for all $n \geq n_0(\xi)$
\begin{align}\label{AbschPatterns}
\Vert\Sigma^n(\xi)\Vert\leq\epsilon_n.
\end{align}
The following index set depends on the dimension $p$, on the inverse temperature $\beta$, the presence or absence of an external magnetic field $h$ and its
direction $e_l$:
\begin{align}\label{L}
L:=\begin{cases}\{\sgn(h)l\}, &\text{in the case }h\neq 0,\\
\{1\},&\text{in the case }0<\beta<\beta_c\text{ and }h=0,\\
\{-p,\ldots,-1,1,\ldots,p\},&\text{in the case }\beta>\beta_c\text{ and }h=0.\end{cases}
\end{align}
The index set $L$ is used to describe those directions that the overlap favors under the equilibrium measure. 
In $(\beta_c,0)$ the central limit theorem fails (see \cite{Gentz:1996}). Thus we do not need $L$ for these parameters.
The following result is proved in \cite[Proposition 2.3]{Gentz:1996} and is an important step for defining the centering.

\begin{prop}\label{ProPhi}~\\
Let $\beta>0$ and $h\geq 0$ such that $(\beta,h)\neq(\beta_c,0)$ and $l\in\{-p,\ldots,-1,1,\ldots,p\}$. For $\lambda\in\R^p$, we define the $\xi$-dependent function
\begin{align}\label{PHI}
\Phi(\lambda):&=-\frac{1}{2\beta}\left\Vert  \lambda-he_l \right\Vert^2+\frac{1}{n}\sum\limits_{j=1}^n\log\cosh\langle\lambda,\xi_j\rangle.
\end{align}
Then, for all strictly positive $c_1<(1-\beta(1-(x^*)^2))/\beta$, there exists an $r_1>0$, depending on $\beta$, $h$ and $c_1$ only, and for $\Ws_{\xi}$-almost all $\xi$, there exists an $n_1(\xi)\geq n_0(\xi)$, which does not depend on the choice of $l$, such that for all $n\geq n_1(\xi)$ the following assertions hold:
\begin{enumerate}
\item For all $\lambda$ in the closed ball $\overline{B_{r_1}(\arctanh(x^*e_l)})$, the matrix $-D^2\Phi(\lambda)$ is uniformly positive definite in the sense that
\begin{align*}
\langle u,-D^2\Phi(\lambda)u\rangle\geq c_1\Vert u\Vert^2\text{  for all }u\in\R^p.
\end{align*}
\item On the set $\overline{B_{r_1}(\arctanh(x^*e_l)})$, the map $\Phi$ has a unique maximum which is attained in the point $\lambda_l^n(\xi)$ satisfying
\begin{align*}
\vert \lambda_l^n(\xi)-\arctanh(x^*e_l)\vert\leq c_2\epsilon_n
\end{align*}
with $c_2=2|x|/c_1$. In particular, $\lambda_l^n(\xi)=0$ in the case $\beta<\beta_c$ and $h=0$.
\end{enumerate}
\end{prop}

\begin{remark}
The function $\Phi$ defined in \eqref{PHI} is sometimes called {\it quenched free-energy} of the Hopfield model.
If the realizations $\xi_1,\ldots,\xi_n$ take all possible values with the same frequency and $n$ is a multiple of $2^p$, then $\lambda_l^n(\xi)=\arctanh(x^*e_l)$. 
\end{remark}
The random centering is given by
\begin{align}\label{XLN}
x_l^n(\xi)&=\frac{1}{\beta}(\lambda_l^n(\xi)-he_l)
\end{align}
with the help of $\lambda_l^n(\xi)$ for $l\in\{-p,\ldots,-1,1,\ldots,p\}$. Even if it is not indicated by the name it remains important to notice that \eqref{XLN} still depends on $\beta$ and $h$. We have to extend this definition because \eqref{XLN} is only defined for $\Ws_{\xi}$-almost all $\xi$ and $n\geq n_1(\xi)$. We assign
\begin{align}\label{XLN2}
x_l^n(\xi)&=\frac{1}{\beta}(\arctanh x^*-h)e_l=x^*e_l
\end{align}
whenever $\lambda_l^n(\xi)$ is not defined. The second equality of \eqref{XLN2} is due to the Curie-Weiss equation \eqref{CWEh}. Using Proposition \ref{ProPhi} we see that for $\beta<\beta_c$ the centering satisfies $x_l^n(\xi)=0$, while for $\beta>\beta_c$ the centering is close to the limiting point $x^*$ in the sense that
\begin{align}\label{ConvCentring}
\Vert x_l^n(\xi)-x^*e_l\Vert\leq \frac{1}{\beta}c_2\epsilon_n\rightarrow 0
\end{align}
as $n\rightarrow\infty$ for some constant $C$ and $\epsilon_n$ defined in \eqref{epsilonn}. 

From now on we will write random vectors in $\mathbb{R}^d$ in the form $w=(w_1,\ldots,w_d)^t$, where $w_i$ are 
$\mathbb{R}$-valued variables for $i=1,\ldots,d$. If a matrix $\Sigma$ is symmetric, nonnegative definite, we denote by $\Sigma^{1/2}$ the unique 
symmetric, nonnegative definite square root of $\Sigma$. 
$\text{Id}$ denotes the identity matrix and from now on $Z$ will denote a random vector 
having standard multivariate normal distribution. The expectation with respect to the measure $P_{n,\beta,he_l,\xi}$ will be denoted by 
$\E:=\E_{P_{n,\beta,he_l,\xi}}$. 

Let $\pi_k:\R^p\rightarrow\R^k$ (with $k\leq p$) denote the canonical projection.
\begin{theorem}\label{THHUM}~\\
Let $\beta, h> 0$, $l\in\Zz$, $l\neq 0$, and $k\in\N$. We assume that $p$ depends on $n$ in a nondecreasing way satisfying $p\leq n$. Let $x=x_l^n(\xi)$ be defined as in \eqref{XLN} and $W$ be the following random variable:
\begin{align*}
W:=\sqrt{n}\pi_k\left(\frac{S_n}{n}-x\right).
\end{align*}
If $Z$ has the $k$-dimensional standard normal distribution, under the measure $P_{n,\beta,he_l,\xi}$, we have, for every three times differentiable function $g$ and $\Ws_{\xi}$-almost all $\xi$,
\begin{align*}
\big| \E g(W) - \E g \left( \Sigma^{1/2} Z\right) \big| \leq C \max\left\{p\sqrt{p}\epsilon_n,\frac{p^2}{n^{1/2}}\right\},
\end{align*}
for a constant $C$ and $\Sigma := \E\left[W \, W^t \right]$.
\end{theorem}

\begin{remark}
The rate of convergence obtained here is useless unless 
\begin{align}\label{RateHop}
\max\left\{p\sqrt{p}\epsilon_n,\frac{p^2}{n^{1/2}}\right\}\rightarrow 0.
\end{align}
In \cite[Theorem 1.1]{Bovier/Gayrard:1997} 
the authors proved that the condition $p/n\rightarrow 0$ is sufficient in order to state the central limit theorem and show the weak convergence. 
In \cite{Gentz:1996} and \cite{Bovier/Gayrard:1997}  there is no information available on the speed of convergence.
Obviously \eqref{RateHop} is poorer but we do not need any conditions on $p$ in advance. Our theorem implies weak convergence.
\end{remark} 

In order to state a result for non-smooth test functions $g$ in the multivariate setting, we introduce
a class of test functions $\G$ following \cite{Rinott/Rotar:1996}. Let again $\Phi$ denote the standard normal distribution function in $\R^d$. We define for $g:\R^d\rightarrow \R$
\begin{align}
g_{\delta}^+(x)&=\sup\{g(x+y):|y|\leq\delta\},\label{g+}\\
g_{\delta}^-(x)&=\inf\{g(x+y):|y|\leq\delta\},\label{g-}\\
\tilde g(x,\delta)&=g_{\delta}^+(x)-g_{\delta}^-(x)\label{gschl}.
\end{align}
Let $\G$ be a class of real measurable functions on $\R^d$ such that
\begin{enumerate}
\item The functions $g\in\G$ are uniformly bounded in absolute value by a constant, which we take to be 1 without loss of generality.
\item For any $d\times d$ matrix $A$ and any vector $b\in\R^d$, $g(Ax+b)\in\G$.
\item For any $\delta>0$ and any $g \in \G$, $g_{\delta}^+(x)$ and $g_{\delta}^-(x)$ are in $\G$.
\item For some constant $a=a(\G,d)$, $\sup\limits_{g\in\G}\left\{\int\limits_{\R^d}\tilde g(x,\delta)\Phi(dx)\right\}\leq a\delta$. 
\end{enumerate}
Obviously we may assume $a\geq 1$. Considering the one dimensional case, we notice that the collection of indicators of all half lines and indicators of all intervals form classes in $\G$ 
that satisfy these conditions with $a=\sqrt{2/\pi}$ and $a=2\sqrt{2/\pi}$ respectively. This was shown for example in \cite{Rinott/Rotar:1997}. 
In dimension $d \geq 1$ the class of indicators of convex sets is known to be such a class.
Using this class of functions we are able to present rates of convergence for non-smooth test functions.

\begin{theorem} \label{THHUM2}~\\
Let $\beta, h> 0$, $l\in\Zz$, $(l\neq 0)$ and $k\in\N$. We assume that $p$ depends on $n$ in a nondecreasing way satisfying $p\leq n$. Let $x=x_l^n(\xi)$ be defined as in \eqref{XLN} and $W$ be as in Theorem \ref{THHUM}. If $Z$ has the $k$-dimensional standard normal distribution, under the measure $P_{n,\beta,he_l,\xi}$, we have, for all $g\in \G$ with $|g|\leq 1$ and $\Ws_{\xi}$-almost all $\xi$,
\begin{align*}
\big| \E g(W) - \E g \left( \Sigma^{1/2} Z\right) \big| \leq C\log(n)\max\left\{p\sqrt{p}\epsilon_n,\frac{p^2}{n^{1/2}}\right\},
\end{align*}
for a constant $C$ and $\Sigma := \E\left[W \, W^t \right]$.
\end{theorem}

In the case where $p$ is fixed the rate gets much simpler since we do not need the projection in order to reduce the size of the vector $W$.

\begin{theorem}\label{THHUMFix}~\\
Let $\beta, h> 0$, $l\in\Zz$ and $l\neq 0$. We assume that $p$ is fixed. Let $x=x_l^n(\xi)$ be defined as in \eqref{XLN} and $W$ be the following random variable:
\begin{align*}
W:=\sqrt{n}\left(\frac{S_n}{n}-x\right).
\end{align*}
If $Z$ has the $p$-dimensional standard normal distribution, under the measure $P_{n,\beta,he_l,\xi}$, we have, for every three times differentiable function $g$ and $\Ws_{\xi}$-almost all $\xi$,
\begin{align*}
\big| \E g(W) - \E g \left( \Sigma^{1/2} Z\right) \big| \leq C n^{-1/2},
\end{align*}
for a constant $C$ and $\Sigma := \E\left[W \, W^t \right]$.
\end{theorem}
With the same techniques necessary to prove Theorem \ref{THHUM2} we get a theorem similar to Theorem \ref{THHUMFix} with rate $\log(n)n^{-1/2}$.

When there is no external field it is natural to ask for the fluctuations of the overlap around $x^*e_l$. With $L$ as in \eqref{L} we determine the conditional fluctuations and a rate of convergence:

\begin{theorem}\label{THHMM}~\\
Let $\beta >0$, $\beta\neq \beta_c$, $h=0$, $l\in L$ and $k\in\N$. We assume that $p$ depends on $n$ in a nondecreasing way satisfying $p\leq n$. Let $x=x_l^n(\xi)$ be defined as in \eqref{XLN} and $W$ be as in Theorem \ref{THHUM}. Then, if $Z$ has the $k$-dimensional standard normal distribution, under the conditional measure
\begin{align*}
P_{n,\beta,\xi}\left(\,\,\cdot\,\,\bigg| \frac{S_n}{n}\in B(x^*e_l,\epsilon)\right),
\end{align*}
we have for every three times differentiable function $g$ and $\Ws_{\xi}$-almost all $\xi$,
\begin{align*}
\big| \E g(W) - \E g \left( \Sigma^{1/2} Z\right) \big| \leq C \max\left\{p\sqrt{p}\epsilon_n,\frac{p^2}{n^{1/2}}\right\},
\end{align*}
for a constant $C$ and $\Sigma := \E\left[W \, W^t \right]$.
\end{theorem}

Note that also for the case of $h=0$ a theorem for non-smooth test functions could be stated, similar to Theorem \ref{THHUM2}, and additionally we obtain a theorem if $p$ is fixed with rate $n^{-1/2}$ in the same way as in Theorem \ref{THHUMFix}.

In Section 2 of the present paper, we introduce Stein's method and present two plug-in theorems for multivariate normal approximation.
Section 3 contains some auxiliary results which will be necessary for the proofs given in Section 4. 

\section{Stein's method of exchangeable pairs}

Starting with a bound for the distance between univariate random variables and the normal distribution Stein's method was first 
published in \cite{Stein:1972} (1972). In \cite{Stein:1986} Stein introduced his exchangeable pair approach. At the heart of the method is a coupling of a random variable $W$ with 
another random variable $W'$ such that $(W,W')$ is {\it exchangeable}, i.e. their joint distribution is symmetric. 
Stein proved further on that a measure of proximity of W to normality may be provided by the exchangeable pair if $W'-W$ is sufficiently small. 
He assumed the property that there is a number $\lambda>0$ such that the expectation of $W'-W$ with respect to W satisfies
$$
\E[W'-W| W]=-\lambda W.
$$
Heuristically, this condition can be understood as a linear regression condition: if $(W,W')$ were bivariate normal with correlation $\varrho$, then
$\E [W' |W] = \varrho \, W$ and the condition would be satisfied with $\lambda = 1 - \varrho$.
Stein proved that for any uniformly Lipschitz function $h$
\begin{align*}
| \E h(W) - \E h(Z) | \leq \delta \|h'\|
\end{align*}
with $Z$ denoting a standard normally distributed random variable and
\begin{equation*} 
\delta = 4 \E \bigg| 1 - \frac{1}{2 \lambda} \E \bigl[ (W'-W)^2 | W \bigr] \bigg| + \frac{1}{2 \lambda} \E |W-W'|^3.
\end{equation*}
Stein's approach has been successfully applied in many models, see e.g. \cite{Stein:1986} or
\cite{DiaconisStein:2004} and references therein. In \cite{Rinott/Rotar:1997} the range of application was
extended by replacing the linear regression property by a weaker condition assuming that there is also a random variable $R=R(W)$ such that
$$
\E[W'-W| W]=-\lambda W+R.
$$
While the approach has proved successful also in non-normal contexts (see \cite{ChatterjeeDiaconis:2005},\cite{Chatterjee/Shao:2009} and \cite{Eichelsbacher/Loewe:2010}) 
it remained restricted to the one-dimensional setting for a long time. 
 Applying the linear regression heuristic in the multivariate case leads to a new condition due to \cite{ReinertRoellin:2009}:
\begin{equation} \label{regressioncond}
\E[W'-W| W]=-\Lambda W+R
\end{equation}
for an invertible $d\times d$ matrix $\Lambda$ and a remainder term $R=R(W)$. 
Different exchangeable pairs, obviously, will yield different $\Lambda$ and $R$. 

The theorems for smooth test functions are based on a nonsingular multivariate normal approximation theorem taken
from \cite{ReinertRoellin:2009}. 
To present this theorem we fix some more notations. The transpose of the inverse of a matrix will be presented in the form $A^{-t}:=(A^{-1})^t$. 
Furthermore we will need the supremum norm, denoted by $\Vert\cdot\Vert$ for both functions and matrices. 
For derivatives of smooth functions $f: \R^d \to \R$, we use the notation $\nabla$ for the gradient operator.
For a function $f:\R^d\to \R$, we abbreviate 
$$
|f|_1 := \sup\limits_{i}\bigl\Vert\frac{\partial}{\partial x_i}f\bigr\Vert, \quad |f|_2:=
\sup\limits_{i,j}\bigl\Vert\frac{\partial^2}{\partial x_i\partial x_j}f\bigr\Vert,
$$
and so on, if these derivatives exist. 

\begin{theorem} \label{RR}(Reinert, R\"ollin: 2009)\\
Assume that $(W,W')$ is an exchangeable pair of $\R^d$-valued random vectors such that
$$
\E[W]=0, \quad \E [W \, W^t] =\Sigma,
$$
with $\Sigma \in \R^{d\times d}$ symmetric and positive definite. If $(W,W')$ satisfies \eqref{regressioncond} for an invertible matrix $\Lambda$ and a $\sigma(W)$-measurable random vector $R$ and if $Z$ has d-dimensional standard normal distribution, 
we have for every three times differentiable function $g$,
\begin{equation} \label{mainboundRR}
 \left|\E g(W) - \E g\left(\Sigma^{1/2} Z\right) \right| \leq \frac{|g|_2}{4}A+\frac{|g|_3}{12}B+\left(|g|_1+\frac{1}{2}d\Vert\Sigma\Vert^{1/2}|g|_2\right)C,
\end{equation}
where, with $\lambda^{(i)}:= \sum\limits_{m=1}^d \left|(\Lambda^{-1})_{m,i} \right|$,
\begin{eqnarray*}
A &= &\sum\limits_{i,j=1}^d \lambda^{(i)} \sqrt{\V \bigl[\E[(W'_i-W_i)(W'_j-W_j)\mid W]\bigr]},  \\ 
B&=&\sum\limits_{i,j,k=1}^d\lambda^{(i)}\E |(W'_i-W_i)(W'_j-W_j)(W'_k-W_k) |,\\
C&=&\sum\limits_{i=1}^d\lambda^{(i)}\sqrt{\V\left[R_i\right]}. 
\end{eqnarray*}
\end{theorem}

The advantage of Stein's method is that the bounds to a multivariate normal distribution reduce to the 
computation of, or bounds on, low order moments, here bounds on the absolute third moments, on a conditional variance and on the variance 
of the remainder term. Such variance computations may be difficult, but we will get rates of convergence at the same time.\\
In the same context as in \cite{ReinertRoellin:2009} the authors in \cite{Eichelsbacher/Martschink:2011} proved the following theorem, presenting bounds for non smooth test functions. Their development differs from \cite{ReinertRoellin:2009} using the relationship to the bounds in \cite{Rinott/Rotar:1997}.

\begin{theorem}\label{Kolmogorov}~\\
Let $(W,W')$ be an exchangeable pair with $\E[W]=0$ and $\E[WW^t]=\Sigma$ with $\Sigma\in\R^{d\times d}$ symmetric and positive definite. Again we assume that $(W,W')$ satisfies \eqref{regressioncond} for an invertible matrix $\Lambda$ and a $\sigma(W)$-measurable random vector $R$ and additionally, for $i\in\{1,\ldots,d\}$, $|W_i'-W_i|\leq A$. Then,
\begin{eqnarray*}
\sup_{g\in\G} |\E g(W)-\E g(\Sigma^{1/2}Z)| &\leq& C \bigl[
\log(t^{-1})A_1+\left(\log(t^{-1})\Vert\Sigma\Vert^{1/2}+1\right)A_2\\
&&+\biggl(1+\log(t^{-1})\sum\limits_{i=1}^d\E|W_i|+a\biggr)A^3A_3
+aA\bigr],
\end{eqnarray*}
where
\begin{eqnarray*}
A_1 &= & \sum\limits_{i,j=1}^d|(\Lambda^{-1})_{j,i}|\sqrt{\V\bigl[\E[(W'_i-W_i)(W'_j-W_j)| W]\bigr]},\\
A_2&=&\sum\limits_{i,j=1}^d|(\Lambda^{-1})_{j,i}|\sqrt{\E\left[R_i^2\right]}, \qquad
A_3=\sum\limits_{i=1}^d\max\limits_{j\in\{1,\ldots,d\}}|(\Lambda^{-1})_{j,i}|,
\end{eqnarray*} 
C denotes a constant that depends on $d$, $\sqrt{t}=2CA^3A_3$ and $a>1$ is taken from the conditions on $\G$, defined before Theorem \ref{THHUM2}.
\end{theorem}

\section{Auxiliary results}
The quenched free-energy $\Phi$ defined in \eqref{PHI} will appear in the regression condition \eqref{regressioncond}.
\begin{lemma}\label{PhiTan}~\\
For $\Phi$ defined in \eqref{PHI} we obtain
\begin{align*}
\frac{1}{n}\sum\limits_{j=1}^n\xi_j^i\tanh(\langle\lambda,\xi_j\rangle)=\frac{1}{\beta}\left(\lambda_i-h\delta_{i,l}\right)+\frac{\partial}{\partial \lambda_i}\Phi(\lambda).
\end{align*}
\end{lemma}
\begin{proof}
Differentiating with respect to $\lambda_i$ yields
\begin{align*}
\frac{\partial}{\partial \lambda_i}\Phi(\lambda)&=\frac{1}{\beta}\left(\lambda_i-h\delta_{i,l}\right)+\frac{1}{n}\sum\limits_{j=1}^n\frac{\sinh(\langle\lambda,\xi_j\rangle)}{\cosh(\langle\lambda,\xi_j\rangle)}\xi_j^i\\
&=\frac{1}{\beta}\left(\lambda_i-h\delta_{i,l}\right)+\frac{1}{n}\sum\limits_{j=1}^n\tanh(\langle\lambda,\xi_j\rangle)\xi_j^i.
\end{align*}
Rearranging the equality yields the result.
\end{proof}

Moreover we consider
\begin{align*}
C_l^n(\xi):=-D^2\Phi(\lambda_l^n(\xi)) = \frac{1}{\beta} {\rm Id}_{\R^p} - \frac 1n \sum_{i=1}^n \cosh^{-2} \bigl( \langle \lambda_l^n(\xi), \xi_i \rangle \bigr) \xi_i \xi_i^t,
\end{align*}
with $\lambda_l^n(\xi)$ are defined in Propostion \ref{ProPhi}.
\begin{lemma}\label{BoundedDerivativesPhi}~\\
Let $\beta>0$ and $h\geq 0$ such that $(\beta,h)\neq(\beta_c,0)$. Choose an $l\in\Zz$, $l\neq 0$, satisfying $|l|\leq p$ in the case of bounded $p$. Then there exists a constant $c_3>0$ such that
\begin{align*}
\sup\limits_{l\in L}\bigl\Vert C_l^n(\xi)-\frac{1}{\beta}[1-\beta(1-(x^*)^2]{\rm Id}_{\R^p}\bigr\Vert\leq c_3\sqrt{p}\epsilon_n
\end{align*}
for $\Ws_{\xi}$-almost all $\xi$ and all $n\geq n_1(\xi)$.
\end{lemma}
Here $\Vert\cdot\Vert$ denotes the operator norm. The proof of Lemma \ref{BoundedDerivativesPhi} is given in \cite[Lemma 3.2]{Gentz:1996} and uses \eqref{AbschPatterns}, Proposition \ref{PhiTan} and that with \eqref{CWEh} $x^*$ satisfies $\cosh^{-2}\arctanh x^*=1-(x^*)^2$. 

Using the notation 
\begin{align}\label{mi} 
m_i^j(\sigma,\xi) &:= \frac 1n \sum\limits_{\mu=1}^p\sum\limits_{\stackrel{r=1}{r \not= j}}^n \xi_i^{\mu}\xi_r^{\mu}\sigma_r,\\\label{mik}
m_{i}(\sigma,\xi) &:= \frac 1n \sum\limits_{\mu=1}^p\sum\limits_{r =1}^n \xi_i^{\mu}\xi_r^{\mu}\sigma_r
\end{align}
the next lemma states an exact expression for the conditional probability that will occur in the linear regression condition \eqref{regressioncond}.

\begin{lemma} \label{ConDH}~\\
Let $\sigma_i\in\{-1,1\}$. Then we obtain for the conditional distribution of a single spin
\begin{align*}
P_{n,\beta,he_l,\xi}(\sigma_i=t\mid(\sigma_k)_{k\neq i})=\frac{\exp(\beta m_i^i(\sigma,\xi)t+h\xi_i^lt)}{\sum\limits_{k\in\{-1,1\}}\exp(\beta m_i^i(\sigma,\xi)k+h\xi_i^lk)}
\end{align*}
and thus
\begin{align*}
\E[\sigma_i\mid(\sigma_k)_{k\neq i}]=\tanh(\beta m_i^i(\sigma,\xi)+h\xi_i^l),
\end{align*}
where $\E$ denotes the expectation with respect to $P_{n,\beta,he_l,\xi}$.
\end{lemma}
\begin{proof}
Direct calculations yield\\~\\
$P_{n,\beta,he_l,\xi}(\sigma_i=t\mid(\sigma_k)_{k\neq i})$
\begin{align*}
&=\frac{P_{n,\beta,he_l,\xi}(\{\sigma_i=t\}\cap(\sigma_k)_{k\neq i})}{P_{n,\beta,he_l,\xi}((\sigma_k)_{k\neq i})}\\
&=\frac{\exp\biggl[\frac{\beta}{2n}\sum\limits_{\mu=1}^p(\xi_i^{\mu})^2+\frac{2\beta}{2n}\sum\limits_{\mu=1}^p\sum\limits_{\stackrel{j=1}{j\neq i}}^n\xi_i^{\mu}\xi_j^{\mu}\sigma_jt+\frac{\beta}{2n}\sum\limits_{\mu=1}^p\sum\limits_{\stackrel{k,j=1}{k,j\neq i}}^n\xi_k^{\mu}\xi_j^{\mu}\sigma_j\sigma_k+h\xi_i^lt+h\sum\limits_{\stackrel{j=1}{j\neq i}}^n\xi_j^l\sigma_j\biggr]}{\sum\limits_{k\in\{-1,1\}}\exp\biggl[\frac{\beta p}{2n}+\frac{2\beta}{2n}\sum\limits_{\mu=1}^p\sum\limits_{\stackrel{j=1}{j\neq i}}^n\xi_i^{\mu}\xi_j^{\mu}\sigma_jk+\frac{\beta}{2n}\sum\limits_{\mu=1}^p\sum\limits_{\stackrel{k,j=1}{k,j\neq i}}^n\xi_k^{\mu}\xi_j^{\mu}\sigma_j\sigma_k+h\xi_i^lk+h\sum\limits_{\stackrel{j=1}{j\neq i}}^n\xi_j^l\sigma_j\biggr]}\\
&=\frac{\exp(\beta m_i^i(\sigma,\xi)t+h\xi_i^lt)}{\sum\limits_{k\in\{-1,1\}}\exp(\beta m_i^i(\sigma,\xi)k+h\xi_i^lk)},
\end{align*}
where we canceled equivalent expressions in numerator and denominator and used the expression for $m_i^i(\sigma,\xi)$. Thus
\begin{align*}
\E[\sigma_i\mid(\sigma_k)_{k\neq i}]&=P(\{\sigma_i=1\}\cup(\sigma_k)_{k\neq i})-P(\{\sigma_i=-1\}\cup(\sigma_k)_{k\neq i})\\
&=\frac{\exp(\beta m_i^i(\sigma,\xi)+h\xi_i^l)-\exp(-\beta m_i^i(\sigma,\xi)-h\xi_i^l)}{\exp(\beta m_i^i(\sigma,\xi)+h\xi_i^l)+\exp(-\beta m_i^i(\sigma,\xi)-h\xi_i^l)}\\
&=\tanh(\beta m_i^i(\sigma,\xi)+h\xi_i^l).
\end{align*}
\end{proof}

Higher order moments of the rescaled empirical spin vector of the Hopfield model, appearing in Theorems \ref{THHUM} up to \ref{THHMM}, can be bounded as follows:
\begin{lemma}\label{EndlMoH}~\\
For $W$ as in Theorems \ref{THHUM} up to  \ref{THHMM} we obtain that for any $l\in \N$ and $j\in\{1,\ldots,p\}$
\begin{align*}
\E \big| W_j^l \big| \leq \text{const.} (l). 
\end{align*}
\end{lemma}
\begin{proof}
First we will have to make a transformation with the well-known Hubbard-Stratonovich approach, expressing the distribution
of $S_n$ in the Hopfield model in terms of $\Phi$. This approach was for example used in \cite[Lemma 2.2]{Bovier/Gayrard/Picco:1994} and in \cite{Eichelsbacher/Loewe:2004}.
Let ${\rm Id}$ denote the $p\times p$ identity matrix and for $\beta >0$ and $h\geq 0$ we pick a random vector $V$ in a way 
that $\Ll(V)$ equals a $p$-dimensional centered Gaussian vector with covariance matrix $\beta^{-1} {\rm Id}$ and $V$ is chosen to be independent from all other random variables involved. Additionally $\lambda:=\lambda_l^n(\xi)$ denotes the maximum point of $\Phi$ taken from Proposition \ref{ProPhi}. First we note that
\begin{align*}
P_{n,\beta,he_l,\xi}\left( S_n\in \text{d}y\right)=Z_{n,\beta,he_l,\xi}^{-1}\exp\left(\frac{\beta}{2n}\langle y,y\rangle+\langle y,he_l\rangle\right)P_n(S_n\in \text{d}y),
\end{align*}
where $P_n(S_n\in \text{d}y)=\prod\limits_{i=1}^n\rho(\text{d}\sigma_i)$ and $\rho(\text{d}\sigma_i)=\frac 12\delta_{-1}(\text{d}\sigma_i)+\frac 12\delta_{1}(\text{d}\sigma_i)$. Furthermore for $u\in\R^p$ we have
\begin{align}\nonumber
\int\limits_{\R^p}\exp\left(\frac{\beta}{n}\langle u,y\rangle+\langle y,he_l\rangle\right)P_n(S_n\in \text{d}y) 
=\int\limits_{\R^p}\exp\left(\frac{\beta}{n}\sum\limits_{\mu=1}^p\sum\limits_{j=1}^n\xi_j^{\mu}\sigma_ju_{\mu}+\sum\limits_{\mu=1}^p\sum\limits_{j=1}^n\xi_j^{\mu}\sigma_jhe_l^{\mu}\right)\prod\limits_{i=1}^n\rho(\text{d}\sigma_i)\\\nonumber
=\prod\limits_{i=1}^n\int\limits_{\R}\exp\left(\frac{\beta}{n}\langle \xi_i^{\cdot}\sigma_i,u\rangle+\langle \xi_i^{\cdot}\sigma_i,he_l\rangle\right)\rho(\text{d}\sigma_i) 
=\exp\left(\sum\limits_{i=1}^n\log\cosh\langle \xi_i^{\cdot},\frac{\beta u}{n}+he_l\rangle\right).
\end{align}
Hence, for $t\in\R$, $x:=x_l^n(\xi)$ and $A(n)=\sqrt{n}t+nx$ we obtain
\begin{align*}
P\biggl(V+\sqrt{n}&\left(\frac{S_n}{n}-x\right)\leq t\biggr)\\
&=P(\sqrt{n}V+S_n\leq A(n))\\
&=Z_{n,\beta,he_l,\xi}^{-1}\int\limits_{\R^p}\exp\left(\frac{\beta}{2n}\langle y,y\rangle+\langle y,he_l\rangle\right)\\&~~~~~~~~~~\cdot\int\limits_{v\leq A(n)-y}\left(\frac{\beta}{2\pi n}\right)^{p/2}\exp\left(-\frac{\beta}{2n}\langle v,v\rangle\right)dvP_n(S_n\in \text{d}y)
\end{align*}
The substitution $u=v+y$ and abbreviating $C_{p,n}:=Z_{n,\beta,he_l,\xi}^{-1}\left(\frac{\beta}{2\pi n}\right)^{p/2}$ yields
\begin{align*}
P(V&+\sqrt{n}\left(\frac{S_n}{n}-x\right)\leq t)\\
&=C_{p,n}\int\limits_{\R^p}\exp\left(\langle y,he_l\rangle\right)\int\limits_{u\leq A(n)}\exp\left(-\frac{\beta}{2n}\langle u,u\rangle\right)\exp\left(\frac{\beta}{n}\langle u,y\rangle\right)duP_n(S_n\in \text{d}y).
\end{align*}
The abbreviation $\tilde C_{p,n}=C_{p,n}n^{p/2}$ yields
\begin{align*}
P\biggl(V&+\sqrt{n}\left(\frac{S_n}{n}-x\right)\leq t\biggr)\\
&=C_{p,n}\int\limits_{u\leq A(n)}\exp\left(-\frac{\beta}{2n}\langle u,u\rangle+\sum\limits_{i=1}^n\log\cosh\langle \xi_j^{\cdot},\frac{\beta u}{n}+he_l\rangle\right)du\\
&=\tilde C_{p,n}\int\limits_{z\leq t}\exp\biggl(-\frac{\beta}{2n}\langle\sqrt{n}z+n\lambda-nhe_l, \sqrt{n}z+n\lambda-nhe_l\rangle\\
&~~~~+\sum\limits_{i=1}^n\log\cosh\langle \xi_j^{\cdot},\frac{\beta z}{\sqrt{n}}+\lambda-he_l+he_l\rangle\biggr)dz\\
&=\tilde C_{p,n}\int\limits_{z\leq t}\exp\left(n\Phi\left(\frac{\beta z}{\sqrt{n}}+\lambda\right)\right)dz,
\end{align*}
where we used the substitution $u=\sqrt{n}z+nx$ for the second equality. Thus, we have
\begin{align}\label{newden}
\Ll\left(V+\sqrt{n}\left(\frac{S_n}{n}-x\right)\right) 
&= \tilde Z_{n,\beta,he_l,\xi}^{-1} \exp\left[n \Phi\left(\lambda+\frac{\beta x}{n}\right)\right]dx,
\end{align}
where $\tilde Z_{n,\beta,he_l,\xi}^{-1} $ denotes a normalization. Applying this transformation does not change the finiteness of any of the moments of the $W_j$. Thus the new measure has the density \eqref{newden}. Using second-order multivariate Taylor expansion of $\Phi$ (see \eqref{taylorHop}) and the fact that $\lambda$ is a maximum point of $\Phi$ we see that the density of this new measure with respect to the Lebesgue measure is given by 
\begin{align*}
\text{const.} \exp\left[- \frac{1}{2} \langle y,-D^2 \Phi(\lambda) \, y \rangle \right]
\end{align*}
 (up to negligible terms). With Proposition \ref{ProPhi} (a) we know that for any $(\beta,h)\neq (\beta_c,0)$ the Hessian $-D^2 \Phi(\lambda)$ is uniformly positive definite. This fact combined with the transformation of integrals yields that a measure with this density has moments of any finite order.
\end{proof}

\section{Proofs of the Theorems}

Constructing an exchangeable pair in the Hopfield model to obtain an approximate linear regression property \eqref{regressioncond}
leads us to $\Phi$ taken from \eqref{PHI}.
Let $(\beta,h)\neq(\beta_c,0)$, and let $x:=x_l^n(\xi)$ denote the unique global maximum point of $\Phi$, see Proposition \ref{ProPhi}.
For $k\in\N$ fixed, $k\leq p$, we consider
$$
W:=\sqrt{n}\pi_k\left(\frac{S_n}{n}-x\right) = \sqrt{n} \bigl( \frac{1}{n}\sum\limits_{j=1}^n\xi_j^1\sigma_j-x_1, \ldots, \frac{1}{n}\sum\limits_{j=1}^n\xi_j^k\sigma_j-x_k
\bigr)^t.
$$
We start by constructing an exchangeable pair. Therefore we produce a spin collection $\sigma'=(\sigma'_i)_{i\geq 1}$ via a {\it Gibbs sampling procedure}: 
We take $I$ to be a random variable that is uniformly distributed over $\{1,\ldots,n\}$ and independent from all other random variables involved. Exchanging the spin $\sigma_i$ with $\sigma'_i$ drawn from the conditional distribution of the $i^{\text{th}}$ coordinate given $(\sigma_j)_{j\neq i}$ under $P_{n,\beta,he_l,\xi}$, independently from $\sigma_i$, we obtain
\begin{equation} \label{WStrH}
W' := W+\frac{1}{\sqrt{n}}\bigl( \xi_I^1\sigma'_I, \ldots, \xi_I^k\sigma'_I \bigr) -\frac{1}{\sqrt{n}}
\bigl( \xi_I^1\sigma_I, \ldots, \xi_I^k\sigma_I \bigr).
\end{equation}
In this case $(W,W')$ is an exchangeable pair. Let $\F:=\sigma(\sigma_i,\xi_j^{\mu}|i,j,\mu\in\N)$. We obtain that for any $i=1, \ldots, k$:
\begin{eqnarray*}
 \E [ W_i' - W_i | \F ] & = &\frac{1}{\sqrt{n}}\E\left[\xi_I^i\sigma'_I-\xi_I^i\sigma_I| \F\right].
 \end{eqnarray*}
 Using the law of total probability for the conditional expectation and independence we have
 \begin{eqnarray*}
 \E [ W_i' - W_i | \F ]
&=&\frac{1}{\sqrt{n}}\frac{1}{n}\sum\limits_{j=1}^n\E\left[\xi_j^i\sigma'_j-\xi_j^i\sigma_j| \F\right].
\end{eqnarray*}
Since $\sigma_i$ and $\xi_j^i$, $i,j\in\N$, are measurable with respect to $\F$ we obtain
\begin{eqnarray*}
\E [ W_i' - W_i | \F ]
&=&-\frac{1}{\sqrt{n}}\frac{1}{n}S_{n,i}+\frac{1}{\sqrt{n}}\frac{1}{n}\sum\limits_{j=1}^n\xi_j^i\E\left[\sigma_j'| \F\right].
\end{eqnarray*}
With the help of independence and the construction of the exchangeable pair we obtain
$\E\left[\sigma_j'| \F\right]=\E\left[\sigma_j'| \sigma_1,\ldots,\sigma_n\right]=\E\left[\sigma_j| (\sigma_k)_{k\neq j}\right]$.
Applying Lemma \ref{ConDH} yields
\begin{align*}
\E [ W_i' - W_i | \F ] & = -\frac{1}{\sqrt{n}}\frac{1}{n}S_{n,i}+\frac{1}{\sqrt{n}}\frac{1}{n}\sum\limits_{j=1}^n\xi_j^i\tanh(\beta m_j^j(\sigma,\xi)+h\xi_j^l)\nonumber\\
&=-\frac{1}{\sqrt{n}}\frac{1}{n}S_{n,i}+\frac{1}{\sqrt{n}}\frac{1}{n}\sum\limits_{j=1}^n\xi_j^i\tanh(\beta m_j(\sigma,\xi)+h\xi_j^l)+R_{1,i},
\end{align*}
with
\begin{align}
R_{1,i}:=\frac{1}{\sqrt{n}}\frac{1}{n}\sum\limits_{j=1}^n\xi_j^i\left[\tanh(\beta m_j^j(\sigma,\xi)+h\xi_j^l)-\tanh(\beta m_j(\sigma,\xi)+h\xi_j^l)\right].\label{R1iH}
\end{align} 
Now it is important to note that
\begin{align*}
\tanh(\beta m_j(\sigma,\xi)+h\xi_j^l)=\tanh\langle \beta\frac{S_n}{n}+he_l,\xi_j\rangle.
\end{align*}
Thus, with Lemma \ref{PhiTan}, we have
\begin{align*}
\frac{1}{n}\sum\limits_{j=1}^n\xi_j^i\tanh(\beta m_j(\sigma,\xi)+h\xi_j^l)=\frac{1}{\beta}\left(\beta\frac{S_{n,i}}{n}+h\delta_{i,l}-h\delta_{i,l}\right)+\frac{\partial}{\partial \lambda_i}\Phi\left(\beta\frac{S_n}{n}+he_l\right).
\end{align*}
This equation yields
\begin{align}\label{Mot2H}
\E [ W_i' - W_i | \F ] &=\frac{1}{\sqrt{n}}\frac{\partial}{\partial \lambda_i}\Phi\left(\beta\frac{S_n}{n}+he_l\right)+R_{1,i}.
\end{align}
We continue by applying \eqref{XLN} and \eqref{taylorsecondH} (see Appendix) to the first summand in \eqref{Mot2H}.  Since $\lambda_l^n(\xi)$ is a unique maximum point of $ \Phi(\lambda)$ we have $\frac{\partial}{\partial \lambda_i} \Phi (\lambda_l^n(\xi))=0$.
 We also note that $\frac{\beta S_{n,i}}{n}+h\delta_{i,l}-(\lambda_l^n(\xi))_i=\beta\frac{W_i}{\sqrt{n}}$.
Thus, the first summand in \eqref{Mot2H} is equal to 
\begin{align*}
\frac{1}{\sqrt{n}} \sum \limits_{t=1}^k \left( \frac{\partial^2}{\partial \lambda_i\partial \lambda_t} \Phi (\lambda_l^n(\xi))\right)  \, \frac{\beta W_t}{\sqrt{n}} + R_{2,i},
\end{align*}
with 
\begin{equation} \label{R2iH}
R_{2,i} : =\sum \limits_{t=k}^p \left( \frac{\partial^2}{\partial \lambda_i\partial \lambda_t} \Phi (\lambda_l^n(\xi))\right)  \, \frac{\beta W_t}{n}+\sum\limits_{l,t=1}^p\Oo\left(\frac{1}{\sqrt{n}}\frac{W_l}{\sqrt{n}}\frac{W_t}{\sqrt{n}}\right).
\end{equation}
\allowdisplaybreaks
Abbreviating 
\begin{align}\label{RiH}
R(i) &:= R_{1,i} + R_{2,i},
\end{align}
 we have
\begin{eqnarray} 
\E\left[W'_i-W_i\mid \F\right] & = &\frac{1}{n} \sum \limits_{t=1 }^k\left( \frac{\partial^2}{\partial \lambda_i\partial \lambda_t} \Phi (\lambda_l^n(\xi))\right)  \, \beta W_t + R(i) \nonumber \\
& = &  \frac{\beta}{n} \langle \bigl[ D^2 \Phi (\lambda_l^n(\xi)) \bigr]_{i,k} , W \rangle + R(i), 
\end{eqnarray}
where $\langle \cdot, \cdot \rangle$ denotes the Euclidean scalar product and $\bigl[D^2 \Phi(\lambda_l^n(\xi))\bigr]_{i,k}$ denotes the first $k$ entries of the $i^{\text{th}}$ row of the matrix $D^2 \Phi(\lambda_l^n(\xi))$. We obtain
\begin{equation} \label{hzeroH}
\E\left[W'-W\mid \F\right] = \frac{\beta}{ n} \bigl[ D^2 \Phi (\lambda_l^n(\xi)) \bigr]_{\mid k\times k} W + R(W),
\end{equation}
with  $R(W) = (R(1), \ldots, R(k))$. We define $\Lambda := \frac{\beta}{ n} \bigl[ -D^2 \Phi (\lambda_l^n(\xi)) \bigr]_{\mid k\times k}$. With Proposition \ref{ProPhi}(a) $-D^2 \Phi (\lambda_l^n(\xi))$ is uniformly positive definite and thus $\Lambda$ is invertible. We conducted the linear regression condition for the sigma-algebra $\F$ but it should be noted that it yields also the linear regression condition for the sigma-algebra generated by $W$ since $W$ is measurable with respect to $\F$. In this case the linear regression condition \eqref{regressioncond} is fulfilled.

\begin{proof}[Proof of Theorem \ref{THHUM}]
With \eqref{hzeroH} we are able to apply Theorem \ref{RR}. Since the Hessian matrix of $\Phi$ and $\beta$ itself are constants we have $\lambda^{(i)}=\Oo(n)$. We continue by estimating $C$ taken from Theorem \ref{RR}. We start by giving a bound for $R_{1,i}$, defined in \eqref{R1iH}. Since the $\tanh(x)$ is $1$-Lipschitz we obtain
\begin{align*}
|R_{1,i}|&=\left|\frac{1}{\sqrt{n}}\frac{1}{n}\sum\limits_{j=1}^n\xi_j^i\left[\tanh(\beta m_j^j(\sigma,\xi)+h\xi_j^l)-\tanh(\beta m_j(\sigma,\xi)+h\xi_j^l)\right]\right|\\
&\leq\frac{1}{\sqrt{n}}\frac{1}{n}\sum\limits_{j=1}^n\left|\beta m_j^j(\sigma,\xi)+h\xi_j^l-(\beta m_j(\sigma,\xi)+h\xi_j^l)\right|\\
&=\frac{\beta}{\sqrt{n}}\frac{1}{n}\sum\limits_{j=1}^n\left|\frac{1}{n}\sum\limits_{\mu=1}^p\left(\xi_j^{\mu}\right)^2\sigma_j\right|\\
&=\frac{\beta}{\sqrt{n}}\frac{1}{n}\sum\limits_{j=1}^n\left|\frac{1}{n}p\sigma_j\right| \leq\frac{\beta p}{\sqrt{n}}\frac{1}{n}.
\end{align*}
For the estimation of $R_{2,i}$ we note that by Lemma \ref{EndlMoH} we have for the second part of \eqref{R2iH}
\begin{align*}
\E\left[\sum\limits_{l,t=1}^p\Oo\left(\frac{1}{\sqrt{n}}\frac{W_l}{\sqrt{n}}\frac{W_t}{\sqrt{n}}\right)\right]=\Oo\left[\frac{p^2}{n^{3/2}}\right].
\end{align*}
 For the first part of \eqref{R2iH} we note that by Lemma \ref{BoundedDerivativesPhi}, since $i\notin\{k+1,\ldots,p\}$ and $t\in\{1,\ldots,k\}$,
 \begin{align*}
 \left|\frac{\partial^2}{\partial \lambda_i\partial \lambda_t} \Phi (\lambda_l^n(\xi))\right|&\leq c_3\sqrt{p}\epsilon_n
 \end{align*}
 since this expression is a non-diagonal entry of the matrix $-C_l^n(\xi)$. Thus we obtain that
 \begin{align*}
 \E\left[\sum \limits_{t=k}^p \left( \frac{\partial^2}{\partial \lambda_i\partial \lambda_t} \Phi (\lambda_l^n(\xi))\right)  \, \frac{\beta W_t}{n}\right]=\Oo\left[\frac{p\sqrt{p}\epsilon_n}{n}\right],
 \end{align*}
 and finally
\begin{align}\label{BoundRH}
\E |R_{2,i}|=\Oo\left[\max\left\{\frac{p\sqrt{p}\epsilon_n}{n},\frac{p^2}{n^{3/2}}\right\}\right].
\end{align}
 Thus we have 
 \begin{align*} C=\sum\limits_{i=1}^k\lambda^{(i)}\sqrt{\E\left[R(i)^2\right]}=\Oo\left[\max\left\{p\sqrt{p}\epsilon_n,\frac{p^2}{n^{1/2}}\right\}\right].
 \end{align*}
The next thing we notice is that for all $i\in\{1,\ldots,k\}$
\begin{align*}
|W_i'-W_i| =\bigg| \frac{1}{\sqrt{n}}\xi_I^i(\sigma'_I-\sigma_I) \bigg| \leq \frac{1}{\sqrt{n}}.
\end{align*}
 We easily obtain that the bound $B=\Oo(n^{-1/2})$. The only thing left to do is to calculate the tedious conditional variance in $A$. We have:
 \begin{eqnarray} \label{ais}
 \E[(W_i'-W_i)(W_j'-W_j) \mid \F]&=&\frac{1}{n^3}\sum\limits_{t,r=1}^n\xi_t^i\sigma_t\xi_r^j\sigma_r+\frac{1}{n^3}\sum\limits_{t,r=1}^n\E[\xi_t^i\sigma'_t\xi_r^j\sigma'_r\mid \F] \nonumber\\
 &-& \frac{2}{n^3}\sum\limits_{t,r=1}^n\xi_r^j\xi_t^i\sigma_r\E[\sigma'_t\mid \F]\nonumber\\
&=:&A_1+A_2+A_3.\label{AiH}
\end{eqnarray}
To bound the variances of these three terms we abbreviate
\begin{align*}
\tilde m_i(\sigma,\xi):=\frac{1}{n}\sum\limits_{t=1}^n\xi_t^i\sigma_t=\frac{1}{\sqrt{n}}W_i+x_i.
\end{align*}
Thus,
\begin{eqnarray*}
\V[A_1]&=&\frac{1}{n^2}\V\bigl[\tilde m_i(\sigma)\tilde m_j(\sigma)\bigr] 
 = \frac{1}{n^2}\V \biggl[ \frac{W_i \, W_j}{n} + \frac{W_i}{\sqrt{n}} x_j + \frac{W_j}{\sqrt{n}} x_i \biggr] \\& \leq & \frac{1}{n^2}\text{const.} \max \biggl\{ \frac{1}{n^2} \V\bigl[W_i \, W_j\bigr], \frac 1n \V\bigl[W_i\bigr] \biggr\}\\
& \leq& \frac{1}{n^2}\frac{\text{const.}}{n^2} \bigl( \E [ W_i^2 W_j^2] + n \E [W_i] \bigr).
\end{eqnarray*}
Using Lemma \ref{EndlMoH} we obtain $\V[A_1]=\Oo(n^{-3})$. For $A_2$ we obtain
\begin{align*}
A_2=\frac{1}{n^3}\sum\limits_{t,r=1}^n\E\left[\xi_t^i\sigma'_t\xi_r^j\sigma'_r|\F\right]=\frac{1}{n}\E\left[\left(\frac{1}{n}\sum\limits_{t=1}^n\xi_t^i\sigma'_t\right)\left(\frac{1}{n}\sum\limits_{r=1}^n\xi_r^j\sigma'_r\right)|\F\right].
\end{align*}
Next we use the identity $\V[X]=\E[X^2]-(\E[X])^2$ for a random variable $X$ and a conditional version of Jensen's inequality 
in order to obtain that $\V\left[A_2\right]\leq \V\left[A_1\right]=\Oo(n^{-3})$, since $\sigma'$ is an identical copy of $\sigma$.
With Lemma \ref{ConDH} we get
\begin{eqnarray}
- A_3/2  & =&  \frac{1}{n^3} \sum\limits_{t,r=1}^n \xi_r^j\sigma_r\, \E[\xi_t^i\sigma'_t\mid \F] \nonumber\\
&=&\frac{1}{n^3}\sum\limits_{t,r=1}^n\xi_r^j\sigma_r \xi_t^i\tanh(m_t^t(\sigma,\xi)+h\xi_t^l) \nonumber\\
&=& \frac{1}{n^3}\sum\limits_{t,r=1}^n\xi_r^j\sigma_r \xi_t^i\left[\tanh(m_t^t(\sigma,\xi)+h\xi_t^l)-\tanh(m_t(\sigma,\xi)+h\xi_t^l)\right]\nonumber\\
&&+\frac{1}{n^3}\sum\limits_{t,r=1}^n\xi_r^j\sigma_r\xi_t^i\tanh(m_t(\sigma,\xi)+h\xi_t^l)\nonumber\\
& = :&M_1+M_2.\label{MiH}
\end{eqnarray}
Using the same estimations as for $R_n^{(1)}(i)$ we obtain
\begin{align*}
M_1 \leq \left|\frac{1}{n^2} \sum \limits_{r=1}^n \xi_r^j\sigma_r\right|  \left|\frac{\beta p}{n}\right| = \left|\frac{1}{n} \, \beta p \biggl(\frac{W_j}{\sqrt{n}}+x_j\biggr)\right|.
\end{align*}
Hence $\V[M_1]= \Oo\left[\frac{p^2}{n^3}\right]$ by Lemma \ref{EndlMoH}. Additionally we get by using Lemma \ref{PhiTan}, \eqref{taylorsecondH} and the abbreviation $\Phi^{(2),i,j}(\lambda):=\frac{\partial^2}{\partial \lambda_i\partial \lambda_t} \Phi (\lambda_l^n(\xi))$
\begin{align*}
M_2&=\frac{1}{n}\left(\frac{W_j}{\sqrt{n}}+x_j\right)\left(\frac{W_i}{\sqrt{n}}+x_i+\frac{\partial}{\partial \lambda_i}\Phi\left(\beta\frac{S_n}{n}+he_l\right)\right)\\
&=\left(\frac{W_j}{n\sqrt{n}}+\frac{x_j}{n}\right)\left(\frac{W_i}{\sqrt{n}}+x_i+\sum \limits_{t=1}^p \left( \Phi^{(2),i,t}(\lambda)\right)  \, \frac{\beta W_t}{\sqrt{n}}+\sum\limits_{l,t=1}^p\Oo\left[\frac{W_lW_t}{n}\right]\right).
\end{align*}
Since we are estimating the variance of the expressions, constant expressions will vanish. Hence using Lemma \ref{EndlMoH} and Lemma \ref{BoundedDerivativesPhi} in the same way as for \eqref{BoundRH} we have 
\begin{align*}
\V[M_2]= \Oo\left[\max\left\{\frac{p^3\epsilon_n^2}{n^3},\frac{p^2}{n^3}\right\}\right].
\end{align*} 
Therefore $\V[A_3]$ can be bounded by $\Oo\left[\max\left\{\frac{p^3\epsilon_n^2}{n^3},\frac{p^2}{n^3}\right\}\right]$. Thus the variance in $A$ of Theorem \ref{RR} can be bounded by 9 times the maximum of the variances of $A_1, A_2, A_3$. Consequently we obtain
\begin{align*}
A=\sum\limits_{i,j=1}^k\lambda^{(i)}\sqrt{\V\left[\E[(W'_i-W_i)(W'_j-W_j)| W]\right]}=\Oo\left[\max\left\{\frac{p^{3/2}\epsilon_n}{n^{1/2}},\frac{p}{\sqrt{n}}\right\}\right]
\end{align*}
and this completes the proof.
\end{proof}
\begin{proof}[Proof of Theorem \ref{THHUM2}]Having seen the proof of Theorem \ref{THHUM} this proof gets very simple. We first note that Theorem \ref{Kolmogorov} can be applied since the regression condition is the same as for Theorem \ref{THHUM}. $A_1$ matches $A$ taken from the same proof and thus $\log(n)A_1=\Oo\left[\log(n)\max\left\{\frac{p^{3/2}\epsilon_n}{n^{1/2}},\frac{p}{n}\right\}\right]$. Using Lemma \ref{EndlMoH} and the estimation of the C-term in \ref{THHUM} we have that the second expression is $\Oo\left[\log(n)\max\left\{\frac{p^{3/2}\epsilon_n}{n^{1/2}},\frac{p}{\sqrt{n}}\right\}\right]$. The same Lemma, $A=\frac{1}{\sqrt{n}}$ and $A_3=\Oo(n)$ yield that the third and fourth expression have the order $\Oo(\log(n)n^{-1/2})$. Thus the theorem is proven.
\end{proof}
\begin{proof}[Proof of Theorem \ref{THHUMFix}]In order to prove the theorem we have to make small adjustments to the proof of Theorem \ref{THHUM}. Using the same techniques as before we arrive at
\begin{align*}
\E\left[W'-W\mid \mathcal{F}\right] = \frac{\beta}{ n} \bigl[ D^2 \Phi (\lambda_l^n(\xi)) \bigr] W + R(W),
\end{align*}
with  $R(W) = (R(1), \ldots, R(p))$, where $R(i)=R_{1,i}+\tilde R_{2,i}$ with $R_{1,i}$ taken from \eqref{R1iH} and
\begin{equation} \label{R2iHf}
\tilde R_{2,i} : =\sum\limits_{l,t=1}^p\Oo\left(\frac{1}{\sqrt{n}}\frac{W_l}{\sqrt{n}}\frac{W_t}{\sqrt{n}}\right).
\end{equation}
This expression is the central difference to the proof of Theorem \ref{THHUM}. Whereas the expression \eqref{R2iH} contained the expression
\begin{align}\label{BoundpFIx}
\sum \limits_{t=k}^p \left( \frac{\partial^2}{\partial \lambda_i\partial \lambda_t} \Phi (\lambda_l^n(\xi))\right)  \, \frac{\beta W_t}{n},
\end{align}
which made us use Lemma \ref{BoundedDerivativesPhi}, \eqref{BoundpFIx} is now part of $\Lambda W$ since $p$ is a constant and we do not need a projection to define $W$. Thus our expression \eqref{R2iHf} contains just the second expression of the right hand side of \eqref{R2iH}. Fortunately this can be estimated using Lemma \ref{EndlMoH}. Thus, without using Lemma \ref{BoundedDerivativesPhi}, the computation of the rate of convergence gets a lot easier. Again it only remains to estimate $A$, $B$ and $C$ taken from Theorem \ref{RR}. We note that $B$ is the same as in Theorem \ref{THHUM}. Thus $B=\Oo(n^{-1/2})$. $R_{1,i}$ is the same as in \eqref{R1iH} and is bounded in the same way as in Theorem \ref{THHUM}. Since $\tilde R_{2,i}$ was part of \eqref{R2iH} and $p$ is fixed we obtain by using Lemma \ref{EndlMoH}
\begin{align}\label{NewBoundRH}
\E|\tilde R_{2,i}|=\Oo(n^{-3/2}).
\end{align}
In comparison to Theorem \ref{THHUM} and the bound in \eqref{BoundRH} we notice that the first part of the maximum is not existent since the expression \eqref{BoundpFIx} is not part of $\tilde R_{2,i}$ and the second part of the maximum is the same as the bound in \eqref{NewBoundRH} with $p$ constant. Using the bound on $R_{1,i}$ and $\tilde R_{2,i}$ we obtain $C=\Oo(n^{-1/2})$. If we split the expectation of the expression $A$ in the same way as in \eqref{AiH} and we note that $A_1$ and $A_2$ are estimated in exact the same way as for the proof of Theorem \ref{THHUM}. Finally we note that for $p$ fixed we can also split $A_3$ as in \eqref{MiH} and that with the same reasons that led to \eqref{NewBoundRH} $\V[M_1]=\V[M_2]=\Oo(n^{-3})$. Hence, $A=\Oo(n^{-1/2})$.
\end{proof}
\begin{proof}[Proof of Theorem \ref{THHMM}]The proof uses the fact that the conditional joint distribution of the $(\sigma_i)_i$, conditioned on the event $\big\{ \left\Vert\frac{S_n}{n}-x^*e_l\right\Vert<\epsilon \bigr\}$, is given by
\begin{align*}
P_{ n,\beta,\xi}( \sigma)=\frac{1}{\tilde Z_{n, \beta,\xi}}\exp\bigl(-\beta H_n(\sigma,\xi)\bigr)\Ei_{B(x^*e_l,\epsilon)}\left(\frac{S_n}{n}\right),
\end{align*}
where $\tilde Z_{n, \beta,\xi}$ denotes a normalization. Thus we are able to follow the lines of the proof of Theorem \ref{THHUM}.
 \end{proof}  

\section{Appendix} 
For the proofs of the theorems for the Hopfield model we need a multivariate {\it second-order Taylor expansion} of $\Phi(\lambda)$ defined in \eqref{PhiTan}. Let us denote by $D^2\Phi(\lambda)$ the Hessian matrix $\{ \partial^2 \Phi(\lambda) / \partial \lambda_i \partial \lambda_j, i,j = 1, \ldots, p\}$ of $\Phi$ at $\lambda$. We obtain
\begin{eqnarray} \Phi(u)& =&   \Phi(\lambda)+\sum\limits_{k=1}^p\frac{\partial}{\partial u_k}\Phi(\lambda)(u_k-\lambda_{k})+\frac{1}{2} \langle (u-\lambda),D^2\Phi(\lambda)\cdot (u-\lambda) \rangle  \nonumber \nonumber\\\label{taylorHop}
&& +\frac{1}{6}\sum\limits_{t,k,j=1}^p \widetilde R_{t,k,j}(u_t-\lambda_{t})(u_k-\lambda_{k})(u_j-\lambda_{j}),
\end{eqnarray}
with $\bigl| \widetilde R_{t,k,j} \bigr| \leq \bigl\Vert \frac{\partial^3}{\partial u_k\partial u_t\partial u_j}\Phi \bigr\Vert$. For any fixed $m\in \{1,\ldots,p\}$ and any $\lambda, u \in \R^p$ it follows that
\begin{eqnarray}
 \frac{\partial}{\partial u_m}\Phi(u) & = & \frac{\partial}{\partial u_m}\Phi(\lambda)+  \sum\limits_{k=1}\frac{\partial^2}{\partial u_k\partial u_m}\Phi(\lambda)(u_k-\lambda_{k})  \nonumber \\
 &&+ \sum\limits_{k,t =1}^p\Oo((u_k-\lambda_{k})(u_t-\lambda_{t})). \label{taylorsecondH}
\end{eqnarray}

\renewcommand{\refname}{References}

\newcommand{\SortNoop}[1]{}\def\cprime{$'$} \def\cprime{$'$}
  \def\polhk#1{\setbox0=\hbox{#1}{\ooalign{\hidewidth
  \lower1.5ex\hbox{`}\hidewidth\crcr\unhbox0}}}
\providecommand{\bysame}{\leavevmode\hbox to3em{\hrulefill}\thinspace}
\providecommand{\MR}{\relax\ifhmode\unskip\space\fi MR }
\providecommand{\MRhref}[2]{%
  \href{http://www.ams.org/mathscinet-getitem?mr=#1}{#2}
}
\providecommand{\href}[2]{#2}

\end{document}